\documentclass[9pt]{article}

\usepackage{amsthm,amsmath,amssymb,epsf, epsfig, color,enumerate}

\usepackage{amssymb}

\newlength{\originalbase}
\newcommand{\spacing}[1]{\setlength{\baselineskip}{#1\originalbase}}

\usepackage{hyperref}
\newcommand{\arx}[1]{\href{http://arxiv.org/abs/#1}{\texttt{arXiv:#1}}}

\graphicspath{{./images/}}
\date{}

\title{Some Instances of Homomesy Among Ideals of Posets}

\begin{document}

\spacing{1.5}

\theoremstyle{plain}
\newtheorem{theorem}{Theorem}
\newtheorem{lemma}[theorem]{Lemma}
\newtheorem{con}[theorem]{Corollary}
\newtheorem{prop}[theorem]{Proposition}
\newtheorem{fact}[theorem]{Fact}
\newtheorem{observation}[theorem]{Observation}
\newtheorem{claim}[theorem]{Claim}

\theoremstyle{definition}
\newtheorem{defin}[theorem]{Definition}
\newtheorem{exm}[theorem]{Example}
\newtheorem{conjecture}[theorem]{Conjecture}
\newtheorem{open}[theorem]{Open Problem}
\newtheorem{problem}[theorem]{Problem}
\newtheorem{question}[theorem]{Question}

\theoremstyle{remark}
\newtheorem{remark}[theorem]{Remark}
\newtheorem{note}[theorem]{Note}

\author{Shahrzad Haddadan \\
\small Department of Computer Science \\[-0.8ex]
\small Dartmouth College\\[-0.8ex] 
\small NewHampshire , U.S.A.\\
\small\tt shahrzad@cs.dartmouth.edu\\
}




\maketitle
\begin{abstract}
Given a permutation $\tau$ defined on a set of combinatorial objects $S$, together with some statistic $f:S\rightarrow \mathbb{R}$, we say that the triple $\langle S, \tau,f \rangle$ exhibits \emph{homomesy} if $f$ has the same average along all orbits of $\tau$ in $S$. This phenomenon was observed by Panyushev  (2007) \cite{Pan} and later studied, named and extended by Propp and Roby (2013) \cite{PR}. 
After Propp and Roby's paper, homomesy has received a lot of attention, and a number of mathematicians have been intrigued by it \cite{Jagged, PE, Hop, Bloom, NCP}.  While seeming ubiquitous, homomesy is often
surprisingly non-trivial to prove. Propp and Roby studied homomesy in the set of ideals in the product of two chains, with two well known permutations, rowmotion and promotion, the statistic being the size of the ideal. In this paper we extend their results to generalized rowmotion and promotion, together with  a wider class of  statistics  in the product of two chains. 
Moreover, we derive similar  results in other simply described posets. We believe that the framework  we set up here can  be fruitful in demonstrating homomesy results  in  ideals of broader classes of posets. 
\end{abstract}


\noindent 
Consider a poset $\cal P$, and let $J({\cal P})$ be the set containing all of the ideals in $\cal P$.
The \textbf{rowmotion} operation, is an operation mapping $J({\cal P})$ to itself, and it has been studied widely by combinatorists  and under various names (Brouwer-Schrijver map \cite{BS}, the Fon-der-Flaass map \cite{F}, the reverse map \cite{Pan}, and Panyushev complementation \cite{Arm}). 
Rowmotion is defined as follows:

\begin{defin}

Given a poset $\cal P$ on the elements of set $\cal S$, and an order ideal $I\in{J(\cal P})$, rowmotion is denoted by $\Phi$\footnote{Propp and Roby use $\Phi_{J}$ to denote rowmotion acting on order ideals and $\Phi_{A}$ for rowmotion acting on antichains.  In this paper, we discuss only actions on order ideals, thus we drop the subscript.}, and it is defined to be the down set
\footnote{In a poset $\cal P$ on elements of $\cal S$ the down set of a set ${\cal X}\subseteq{\cal S}$ is the following: $ \{y\in {\cal S}| \exists x\in {\cal X}, y\leq x\}.$}
 of the minimal elements in ${\cal S}{-}I$.
\end{defin}

Another interesting operation mapping $J({\cal P})$ to itself is the \textbf{toggle map}.  We can define the rowmotion operation also as the combination of several toggles. Toggling is defined as follows:

\begin{defin}
Given poset $\cal P$ on the elements of set $\cal S$,  an order ideal $I\in J({\cal P})$, and an element  $x \in {\cal S}$, the toggle map $\sigma_{x}:J({\cal P})\rightarrow J({\cal P})$ is defined by:

\begin{equation}
   \sigma_{x}(I)=\begin{cases}
  I \cup \{x\}, & \text{if }  x \notin I \text{ and } I  \cup \{x\} \in J({\cal P}) .\\
  I - \{x\}, & \text{if }    x \in I \text{ and } I  - \{x\} \in J({\cal P}) .\\
    I, & \text{otherwise}.
  \end{cases}
\end{equation}

\end{defin}

\begin{prop}\label{toggcom}
 For all $x\in {\cal S}$ and $I\in J({\cal P})$, $\sigma_{x}^2(I)=I$. If $x,y\in {\cal S}$ and $x$ does not cover  $y$  nor $y$ covers $x$, we have $\sigma_{x}\circ \sigma_{y}(I)=\sigma_{y}\circ \sigma_{x}(I)$.
\end{prop}

We take a  linear extension $(x_1,\dots,x_n)$ of ${\cal P}$ to be an indexing of the elements of ${\cal P}$ that is $x_i < x_j$ in ${\cal P}$ implies
$i<j$.
The following proposition was demonstrated in \cite{Cam}.

\begin{prop}\cite{Cam}\label{LER}
Given an arbitrary $I\in J({\cal P})$ and linear extension $(x_1,\dots,x_n)$ of ${\cal P}$, we have 
$\Phi(I)=\sigma_{x_{1}}\circ \sigma_{x_{2}}\circ \sigma_{x_{3}} \circ \cdots \circ \sigma_{x_{n}}(I).$

\end{prop}

\begin{defin}\label{posets}
Let  ${\cal Q}_{a,b}=[a]\times[b]$ ($[n]=\{1,2,\dots,n\}$). Each element of the poset
can be presented by a pair $(i,j), i \in [a] , j\in [b]$ and $(i_{1},j_{1})\leq (i_{2},j_{2})$ iff $i_{1}\leq i_{2}$ and $j_{1}\leq j_{2}$.

In this paper, we are interested in the maps on $J({\cal Q}_{a,b})$, as well as $J({\cal U}_a)$ and $J({\cal L}_a)$ where ${\cal U}_a$ and ${\cal L}_a$ are subsets of ${\cal Q}_{a,a}$ and defined as:
\end{defin}
\begin{itemize}
\item ${\cal U}_a\subseteq {\cal Q}_{a,a}$, ${\cal U}_a=\{ (i,j)| i,j \in[a], i\geq a{+}1{-}j \}$. 

\item ${\cal L}_a\subseteq {\cal Q}_{a,a}$, ${\cal L}_a=\{ (i,j)| i,j \in[a], i\geq j \}$. 
\end{itemize}

\textbf{Notation.} Let $\cal P$ be one of ${\cal Q}_{a,b}$, ${\cal U}_{a}$ or ${\cal L}_a$. By saying $(i,j)\in {\cal P}$ we are referring to the element in $[a]\times [b]$ with coordinates $i$ and $j$. By saying $x=(i_1,j_1)\leq y=(i_2,j_2)$ we mean $x$ is less than $y$ in $\cal P$. To avoid confusion, we \emph{never} use $(i,j)\in{\cal P}$ to indicate $i$ is less than $j$ in the partial order. 

We call ${\cal Q}_{a,b}$ the square lattice or the product of two chains, ${\cal U}_a$ the upper lattice and ${\cal L}_a$ the left lattice. 
Among combinatorists ${\cal U}_a$ is also known as the root poset of type $A_a$ , and ${\cal L}_a$  as the minuscule poset of type $B_a$ or $D_{a+1}$. 
We employ the following terminology:

\begin{defin}\label{defFRC}

Let $\cal P$ be one of ${\cal Q}_{a,b}$, ${\cal U}_{a}$ or ${\cal L}_a$. For any arbitrary $I\in J({\cal P})$,

 We call the set of all points $(i,j)\in {\cal P}$ with constant $i{+}j$ a \textbf{rank}; $R_{c}(I)=\{(i,j)\in I| i{+}j=c \}$.

 We call the set of all points $(i,j)\in {\cal P}$ with constant $i{-}j$ a \textbf{file}; $F_{c}(I)=\{(i,j)\in I| i{-}j=c \}$.

 We call the sets of all points $(i,j)\in {\cal P}$ with constant $i$ a \textbf{column};  $C_{c}(I)=\{(i,j)\in I| i=c \}$.
 
 In the case when no ideal is specified we have $R_{c}=R_{c}( {\cal P})$,  $F_{c}=F_{c}({\cal P})$ and  $C_{c}=C_{c}({\cal P})$; $\cal P$ should be clear from the context.

\end{defin}

\begin{exm}

The following figure shows ${\cal Q}_{5,4}$ and $R_4$ in it,  ${\cal U}_4$ and $F_{-1}$ in it and ${\cal L}_4$ and $C_3$ in it.

\setlength{\unitlength}{5 mm}
\begin{picture}(6,4)
\end{picture}
\begin{picture}(6,4)
\linethickness{0.3mm}
\multiput(0, 0)(1, 0){5}{\line(0, 1){3}}
\multiput(0, 0)(0, 1){4}{\line(1, 0){4}}
\put(0, 2){\circle*{.3}}
\put(1, 1){\circle*{.3}}
\put(2, 0){\circle*{.3}}
\end{picture}  
\begin{picture}(6,4)
\linethickness{0.3mm}
\put(1, 2){\line(0,1){1}}
\put(2, 3){\line(0,-1){2}}
\put(3, 0){\line(0,1){3}}

\put(2, 1){\line(1,0){1}}
\put(3, 2){\line(-1,0){2}}
\put(0, 3){\line(1,0){3}}

\put(0,3){\line(1,-1){3}}
\put(2, 1){\circle*{.3}}
\put(3, 2){\circle*{.3}}
\end{picture}  
\begin{picture}(6,4)
\linethickness{0.3mm}

\put(2, 2){\line(0,1){1}}
\put(1, 1){\line(0,1){2}}
\put(0, 0){\line(0,1){3}}

\put(0, 1){\line(1,0){1}}
\put(0, 2){\line(1,0){2}}
\put(0, 3){\line(1,0){3}}

\put(0,0){\line(1,1){3}}
\put(2, 2){\circle*{.3}}

\put(2, 3){\circle*{.3}}
\end{picture}  

{\footnotesize
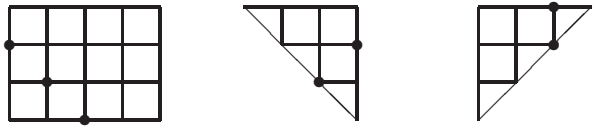
\textbf{Figure 1.} $\quad  \quad\quad \quad  \quad$${\cal Q}_{5,4}$ and $R_4({\cal Q}_{5,4})$ $\quad$  ${\cal U}_4$ and $F_{-1}({\cal U}_4)$ $~ ~ \quad\quad$ ${\cal L}_4$ and $C_3({\cal L}_4)$
}

\end{exm}
We can now define toggling for the above sets.
\begin{defin}
Consider the poset ${\cal Q}_{a,b}$ and $I\in J({\cal Q}_{a,b})$.
Let $S$ be one of $R_{c}$ or $F_{c}$ for some arbitrary $c$. Letting $x_{1}\dots x_{m}$ be some arbitrary indexing of the elements of $S$,
we define $\sigma_{S}(I)=\sigma_{x_{1}}\circ\sigma_{x_{2}}\circ\dots\circ \sigma _{x_{m}}(I)$. Note that no two elements $x_{i}, x_{j}$ of $S$
constitute a covering pair, thus $\sigma_S$ is well defined.
For $S=C_{c}$, let $S=\{x_{1}, x_{2}, \dots x_{m} \}$ where $x_{1}< x_{2}<\dots <x_{m}$. We define $\sigma_{S}(I)=\sigma_{x_{1}}\circ \sigma_{x_{2}}\dots \sigma_{x_{m}}(I)$. 
\end{defin}

Striker and Williams studied the class of so-called rc-posets, whose elements are partitioned into ranks and
files\footnote{Striker and Williams use the terminology ``row'' for what we call ``rank'' and ``column'' for what we call ``file''.}.
Here, we will discuss the special rc-posets of the form ${\cal Q}_{a,b}$, ${\cal U}_a$ or ${\cal L}_a$.  The following definitions are from \cite{SW},
restricted to the product posets of interest to us.

\begin{defin}\cite{SW}
Consider  ${\cal Q}_{a,b}$.
Let $\nu$ be a permutation of $\{2,\dots,a{+}b\}$. We define $\Phi_{\nu}$ to be $\sigma_{R_{\nu(a{+}b{-}1)}}
\circ \sigma_{R_{\nu(a{+}b{-}2)}}\circ \dots \sigma_{R_{\nu(1)}}  $. \end{defin}

Having Proposition, \ref{LER} it can be concluded that  for $\nu =( a{+}b, a{+}b{-}1\dots ,2)$, we have $\Phi_{\nu}=\Phi$.

Consider ${\cal Q}_{a,b}$, and $\nu$  a permutation of $\{2,\dots,a{+}b\}$. Then, $\Phi_{\nu}$ is a permutation on $J({\cal Q}_{a,b})$ that
partitions $J({\cal Q}_{a,b})$ into orbits.  Striker and Williams showed that the orbit structure\footnote{The orbit structure of a bijection $f$ on a set $S$ is the multiset of the sizes of the orbits that bijection $f$ constructs on the set $S$.} of $\Phi_{\nu}$ does not depend on the choice of $\nu$.

\begin{defin}\label{ProDef}

Consider ${\cal Q}_{a,b}$, \textbf{promotion} is a permutation $\partial :J({\cal Q}_{a,b})\rightarrow J({\cal Q}_{a,b})$, defined by:
$\forall I\in J({\cal Q}_{a,b}), \partial(I)= \sigma_{F_{a{-}1}} \circ \sigma_{F_{a{-}2}} \circ \cdots \circ \sigma_{F_{0}} \circ \cdots \circ  \sigma_{F_{1{-}b}}(I)$.

\end{defin}

As with rowmotion, Striker and Williams  \cite{SW} define a generalized version of promotion. 

\begin{defin}\cite{SW}
Consider  ${\cal Q}_{a,b}$,
and let $\nu$ be a permutation of $\{{-}b{+}1,\dots,a{-}1\}$. We define $\partial_{\nu}$ to be $\sigma_{F_{\nu(a{+}b{-}1)}} \circ \sigma_{F_{\nu(a{+}b{-}2)}}\circ \dots  \circ\sigma_{F_{\nu(1)}}  $. By Definition \ref{ProDef},  for $\nu =( {-}b{+}1,\dots ,a{-}1)$ we have $\partial_{\nu}=\partial$.
\end{defin}

As with rowmotion, for any permutation $\nu$ on files of any poset $\cal P$, $\partial_{\nu}$ will partition $J({\cal P})$ to orbits.
Again, Striker and Williams \cite{SW} showed that regardless of which $\nu$ we choose, $J({\cal Q}_{a,b})$ will be partitioned into the same orbit structure
by $\partial_{\nu}$. Moreover, the orbit structures for $\partial_{\nu}$ and for $\Phi_{\omega}$ are the same for any two permutations $\nu$ and $\omega$:

\begin{theorem}\cite{SW}\label{SWthm}
Consider the lattice ${\cal Q}_{a,b}$, for any permutation $\nu$ on $\{2,\dots, a{+}b\}$ and $\omega$ on $\{{-}b{+}1\dots a{-}1\}$,
there is an equivariant bijection between $J({\cal Q}_{a,b})$ under $\Phi_{\nu}$ and $J({\cal Q}_{a,b})$ under $\partial_{\omega}$.
\end{theorem}

Permutations defined on combinatorial structures and the associated orbit structures became more interesting after Propp and Roby introduced a phenomenon called \textbf{homomesy} \cite{PR}. 
 Propp and Roby also discussed some instances of homomesy by studying the actions of promotion
and rowmotion on the set $J({\cal Q}_{a,b})$. 
Homomesy  has attacked many Combinatorics' attentions after it was defined and studied by Propp and Roby \cite{Jagged, PE, Hop, Bloom, NCP}, and it is defined as follows:

\begin{defin}\cite{PR}
Consider a set $S$ of combinatorial objects. Let $\tau:S\rightarrow S$ be a permutation that partitions $S$ into orbits, and $f:S\rightarrow \mathbb{R}$ a statistic of the elements of $S$.
We call the triple $\langle S, \tau, f\rangle$ \textbf{homomesic} (or we say it \textbf{exhibits homomesy}) if and only if there is a constant $c$ such that  for any $\tau$-orbit  $\cal O\subset S$
we have $$\frac{1}{\vert{\cal O}\vert}\sum_{x\in \cal O} f(x)=c.$$ 

Equivalently, we can say $f$ is homomesic  or it exhibits homomesy in $\tau$-orbits of $S$. If $c=0$, the triple is called $0-$mesic.
\end{defin}

\begin{prop}
Consider  a set $S$ and permutation $\tau:S\rightarrow S$.
If $f_{1},\dots, f_{n}$ are homomesic functions in $\tau$-orbits of $S$, then any linear combination of
the $f_{i}$s is also homomesic in $\tau$-orbits of $S$.
\end{prop}

\begin{theorem}\cite{PR}\label{PRthm}
Consider  $f:J({\cal Q}_{a,b})\rightarrow \mathbb{R}$ defined as follows: for all $ I\in {\cal Q}_{a,b}, f(I)= |I|$. Let
$\partial, \Phi: J({\cal Q}_{a,b})\rightarrow J({\cal Q}_{a,b})$ be the rowmotion and promotion operation. The triples
$\langle J({\cal Q}_{a,b}), \partial , f \rangle$ and $\langle J({\cal Q}_{a,b}), \Phi ,f \rangle$ exhibit homomesy.

\end{theorem}

In this paper, we generalize Theorems \ref{PRthm} and \ref{SWthm} in the following sense:

\begin{defin}\label{Comotion}
Consider the poset $\cal P$ to be one of ${\cal Q}_{a,b}$ , ${\cal U}_a$ or ${\cal L}_a$. For any permutation $\nu$ of $[a]$, we define the action \textbf{comotion}, ${\cal T_{\nu}}:{J({\cal P})}\rightarrow {J({\cal P})}$ 
by: $\forall I\in {J({\cal P})}, {\cal T}_{\nu}(I)=\sigma_{C_{\nu(a)}}\circ \sigma_{C_{\nu(a{-}1)}} \circ \dots \circ \sigma_{C_{\nu(1)}}(I) $.

\end{defin}

The following proposition can be proved by applying Proposition \ref{toggcom} inductively. 

\begin{prop}
Let $\cal P$ be one of ${\cal Q}_{a,b}$, ${\cal U}_a$ and ${\cal L}_a$, the action of promotion coincides with ${\cal T}_{(a,a{-}1,\dots 1)}$ and rowmotion coincides with  ${\cal T}_{(1,2,\dots a)}$.

\end{prop}

In what follows, Theorems \ref{main}, \ref{mainU}, \ref{zeromain} which are  the main results of this paper will be stated. We will provide a roadmap to their proofs later in this introduction, and will complete the proof in Sections \ref{next} and \ref{results}. 

\begin{theorem}\label{main}(Homomesy in $J({\cal Q}_{a,b})$)
\end{theorem}

\begin{enumerate}

\item For any arbitrary natural number $a$ and $\nu$ a permutation on $[a]$, ${\cal T}_{\nu}$  partitions $J({\cal Q}_{a,b})$ to orbits.  The orbit structures of ${\cal T}_{\nu}$  on $J({\cal Q}_{a,b})$ is independent of choice of $\nu$. 
\item Consider  $I\in J({\cal Q}_{a,b})$. We have the following homomesy results: \begin{itemize}

\item Let $g_{i,j}, 1\leq i \leq a $ and $1\leq j\leq b$ be defined as follows:

\begin{equation}
g_{i,j}=\begin{cases}
1,& \mbox{if  } \vert C_{i}(I) \vert=j\\
0, &\mbox{otherwise.}
\end{cases}
\end{equation}

 For an arbitrary permutation $\nu$ of  $[a]$, $1\leq i \leq a $ and $0\leq j\leq b$, the function $d_{i,j}= g_{i,j}-g_{a{+}1{-}i, b{-}j}$ is 0-mesic in ${\cal T}_{\nu}$-orbits of $J({\cal Q}_{a,b})$.

\item For all $1 \leq i \leq a$, let

\begin{equation}
s_{i,j}=
\begin{cases} 1 & \mbox{ if }\vert C_{i}(I)\vert+i=j \\
0&\mbox{otherwise.}
\end{cases}
\end{equation}
 For any arbitrary permutation $\nu$ of ${[a]}$ and $1\leq j\leq b$, $s_{j}=\sum_{i=1}^{a}s_{i,j}$  is homomesic in  ${\cal T}_{\nu}$-orbits of $J({\cal Q}_{a,b})$. Moreover, the average of all $s_{j}$ along an orbit is constant and equal to $\frac{a}{a+b}$.
In other words, for all $j, l$,$~s_{l}-s_{j}$ is 0-mesic.
\end{itemize}

Any function $f:J({\cal Q})\rightarrow \mathbb{R}$ which is a linear combination of various $s_{i}$ and $d_{i}$ is homomesic
in ${\cal T}_{\nu}$-orbits of $J({\cal Q}_{a,b})$.

\end{enumerate}

Theorem \ref{main} introduces a different   family of permutations that produce the same orbit structure
as $\Phi$ and $\partial$; hence, it generalized Theorem \ref{SWthm}. It also generalizes Theorem \ref{PRthm} because it introduces a class of permutations and statistics whose triple with  $J({\cal Q}_{a,b})$  exhibit homomesy. Moreover, it will provide another proof for Theorem \ref{PRthm}. The main idea of our proof is the correspondence between comotion and \textbf{winching} (See Definition \ref{Winchdef}). We will define winching and also its  correspondence with comotion  in Section \ref{next}.  Then, we extend the definition of winching  to winching with lower bounds and winching with zeros. Studying these two variations, helps us  obtain homomesy results in $J({\cal U}_a)$ and $J({\cal L}_a)$.

\begin{theorem}\label{mainU}(Homomesy in $J({\cal U}_a)$)
{

Let $a$ be an arbitrary natural number and $\nu$ an arbitrary permutation of $[a]$. Consider ${\cal T}_{\nu}:J({\cal U}_a)\rightarrow J({\cal U}_a)$ as defined in Definition \ref{Comotion}. 
For each $i\in[2a]$ let $[i,2a]={i,i+1,\dots ,2a}$ and $f:[2a]\rightarrow \mathbb{R}$ a function  that has the same average in all $[i,2a]$ where $i$ is odd. Let $g:J({\cal U}_a)\rightarrow \mathbb{R}$ be defined as: $\forall I\in J({\cal U}_a), g(I)=\sum_{i=1}^{a}f(\vert C_i(I)\vert+2i+1)$.
Then, the triple $\langle J({\cal U}_a), {\cal T}_{\nu}, g\rangle$ exhibits homomesy.

}

\end{theorem}

\begin{theorem}\label{zeromain}(Homomesy in $J({\cal L}_a)$)

\end{theorem}

Let $a$ be an arbitrary natural number and  $\nu$  an arbitrary permutation of $[a]$ and ${\cal T}_{\nu}: J({\cal L}_a)\rightarrow J({\cal L}_a)$ be defined as in Definition \ref{Comotion}. We have, 
\begin{enumerate}

\item The orbit structures of ${\cal T}_{\nu}$  on $J({\cal L}_{a,b})$ is independent from choice of $\nu$. 

\item For any $1\leq i\leq a$ and $0\leq j\leq a$ we define $s_{i,j}: J({\cal L}_a)\rightarrow \mathbb{R}$ as follows:

\begin{equation}
s_{i,j}=
\begin{cases} 1 & \mbox{ if }\vert C_{i}\vert=j \\
0&\mbox{otherwise.}
\end{cases}
\end{equation}
For any $ 1\leq j\leq a$ $s_{j}=\sum_{i=1}^{a}s_{i,j}$  is homomesic. Moreover, the average of all $s_{j}$ along any ${\cal T}_{\nu}$-orbit of $J({\cal L}_{a,b})$ is the same.
In other words, for all $j, l$ $~s_{l}-s_{j}$ is 0-mesic.

Moreover, any function $f:J({\cal L}_a)\rightarrow \mathbb{R}$ which is a linear combination of various $s_{i}$ is homomesic
in ${\cal T}_{\nu}$-orbits of $J({\cal L}_a)$.

\end{enumerate}

In Section \ref{next} of this paper we introduce the permutation winching on the set of increasing sequences of length $k$.
We show that there is a natural equivariant bijection between the set of ideals under comotion and the set of increasing sequences under winching.

Then, we introduce two different variations of winching and their correspondence with comotion in $J({\cal U}_a)$ and $J({\cal L}_a)$.

In Section \ref{results} we will use the Theorems \ref{main}, \ref{mainU} and \ref{zeromain} to show homomesy of some functions in the orbit structure produced by comotion in $J({\cal Q}_{a,b})$, $J({\cal U}_a)$ and $J({\cal L}_a)$.
  
In Section \ref{winch} we will prove homomesy  of a class of statistics when the permutation is winching and two different variations of it. These results have intrinsic interest because they are
instances of homomesy. Moreover, 
by the correspondence between winching and comotion proof of 
Theorems \ref{main}, \ref{mainU}, and \ref{zeromain} will be obtained. 
 

\section{Comotion, winching and their correspondence }\label{next}

In the previous section, we defined the action of comotion on the set of order ideals of a poset. 
In this section, we define winching 
and show a correspondence between winching on increasing sequences and comotion on $J({\cal Q}_{a,b})$. 
Then, we define winching with lower bounds and winching with zeros. The former corresponds to comotion on $J({\cal U}_a)$ and the later corresponds to comotion on  $J({\cal L}_a)$. 

\begin{defin}\label{Winchdef}
Let $S_{k,m}$ be the set of all $k$-tuples $x = (x_1,\dots,x_k)$  satisfying $0< x_{1}< x_{2}< \dots< x_{k}< m{+}1$.
We define the map $W_{i}:S_{k,m}\rightarrow S_{k,m}$, called winching on index $i$, by $W_i(x) = y = (y_{1},y_{2},\dots ,y_{k})$, where
$y_{j}=x_{j}$ for $i\neq j$, and

\begin{equation}
 y_{i} = \begin{cases}
 x_{i}+1, & \text{if }  x_{i}+1 < x_{i+1} .\\
x_{i-1}+1, & \text{otherwise}.
  \end{cases}
\end{equation}

We assume that always $x_{0}=0$ and $x_{k{+}1}=m{+}1$. 
\end{defin}

\begin{exm}
Let $\nu =(2,3,1,4)$ and $x\in S_{4,7}$ be $x=(2,3,5,7)$. Then, $W_{\nu}(x)=(1,4,6,7)$.
\end{exm}

\begin{lemma}\label{LemmaAlpha}
There is a bijection $\alpha:J({\cal Q}_{a,b})\rightarrow S_{a,a{+}b}$ such that for any $I\in J({\cal Q}_{a,b})$, $\alpha(\sigma_{C_j}(I))=W_{j}(\alpha(I))$.
\end{lemma}

\begin{proof}
Consider $I\in J({\cal Q}_{a,b})$, we define $\alpha(I)=(\alpha_{1},\dots \alpha_{a})$ as follows: for any $1\leq j\leq a$, we have $ \alpha_{j}(I)=\vert  C_{a{+}1{-}j}(I)   \vert +j$. Since $I\in J({\cal Q}_{a,b})$, for any $j_{1}<j_{2}$, take $\vert C_{j_{1}}(I) \vert \geq \vert  C_{j_{2}}(I) \vert$. Therefore, $\alpha(I)$ is an increasing sequence. 

Let $C_{j}$ be $\{v_{1},v_{2} ,\dots v_{b}\}$; $v_{i}=(j,i)$, and assume $\vert C_{j}(I)\vert=l$.  We have, $ n>l{+}1$, $\sigma_{v_{n}}(I)=I$, and for $n=l{+}1, \sigma_{v_{n}}(I)=I\cup \{v_{n}\}$ if and only if $\vert C_{j{-}1}\vert \geq l{+}1$. Also, $n<l, \sigma_{v_{n}}(I)=I- \{v_{n}\}$ if and only if $\vert C_{j{+}1}(I)\vert\leq n{-}1$. 
For boundary cases, we assume $\vert C_0\vert=b  $ and $\vert C_b\vert=0 $.
Letting $K=\sigma_{C_{j}}(I)$ we will have,

\begin{equation}
 {C_{j}}(K) = \begin{cases}
C_j(I)\cup \{v_{l{+}1}\}, & \text{if } \vert C_{j{-}1}(I)\vert \geq l{+}1 .\\
C_{j}(I)-\{v_{l},v_{l{-}1},\dots, v_{p+1}\} (p=\vert C_{j{+}1}(I)\vert), & \text{otherwise}.
  \end{cases}
\end{equation}

\begin{equation}
\Leftrightarrow  \vert {C_{j}}(K) \vert +a{+}1{-}j= \begin{cases}
l{+}1 +a{+}1{-}j, & \text{if } \vert C_{j{-}1}(I)\vert + a{-}j{+}2\geq l{+}1+ a{-}j{+}2 .\\
\vert C_{j{+}1}(I)\vert +a{+}1{-}j, & \text{otherwise}.
  \end{cases}
\end{equation}

\begin{equation}
\Leftrightarrow \alpha_{a{+}1{-}j}(\sigma_{C_{j}}(I)) = \begin{cases}
\alpha_{a{-}j{+}1}(I){+}1, & \text{if } \alpha_{a{-}j{+}2}(I) > \alpha_{a{-}j{+}1}(I)+1 .\\
\alpha_{a{-}j}(I)+1, & \text{otherwise}.
  \end{cases}
\end{equation}

\begin{equation}
\Leftrightarrow \alpha_{a{+}1{-}j}(\sigma_{C_{j}}(I))= W_{a{+}1{-}j}( \alpha(I)).
\end{equation}
\end{proof}

\begin{con}\label{bijection}
Consider an arbitrary natural number $a$, $\nu$ a permutation of  $[a]$, and for any $x\in S_{a,a+b}$, let $W_{\nu}(x)= W_{\nu(a)}\circ W_{\nu(a{-}1)}\circ\dots \circ W_{\nu(1)}(x)$.
 The bijection $\alpha$ introduced in Definition \ref{LemmaAlpha} satisfies the following property: $\alpha({\cal T_{\nu}}(I))=W_{\nu} (\alpha(I))$.
\end{con}

\begin{theorem}\label{winchthm}
Consider a natural number $k$ and an arbitrary permutation $\nu$ of $[k]$. With $W_{\nu}: S_{k,m}\rightarrow S_{k,m}$ defined as above we will have,
\begin{enumerate}

\item $W_{\nu}^m(x)=x$ for all $x\in S_{k,m}$.
\item The orbit structure that  winching produces on the set $S_{k,m}$ is the
same as the orbit structure for rotation acting on the set of $2$-colored necklaces with $k$ white beads and $n-k$ black beads, and hence independent of choice of $\nu$.
\footnote{The definition of rotation acting on the set of $2$-colored necklaces is presented in Section \ref{winch} (Definition \ref{Neck}).}
 (The orbit structure of necklaces if a classical  problem in Combinatorics and the solution is a result of applying P$\acute{o}$lya's Theorem \cite{Polya}.)

\item The following functions (and any linear combination of them) are homomesic in $W_{\nu}$-orbits of $S_{k,m}$.

\begin{itemize}

\item Let $g_{i,j}:S_{k,m}\rightarrow \mathbb{R}, 1\leq i \leq k $ and $1\leq j\leq m$ be defined as follows:

\begin{equation}
g_{i,j}(x)=\begin{cases}
1,& \mbox{if  } x_{i}=j\\
0, &\mbox{otherwise.}
\end{cases}
\end{equation}

For any arbitrary $1\leq i \leq k $ and $1\leq j\leq m$, the function $d_{i,j}= g_{i,j}-g_{k{+}1{-}i, m{+}1{-}j}$ is 0-mesic in  $W_{\nu}$-orbits of $S_{k,m}$.

\item  For an arbitrary $1\leq  j\leq m$, let $f_j:S_{k,m}\rightarrow \mathbb{R}$ be defined by:

\begin{equation}
f_{j}(x)=\begin{cases}
1,& \mbox{if } j\in x\\
0,& \mbox{otherwise.}
\end{cases}
\end{equation}

For any $1\leq j\leq m$, the triple $\langle S_{k,m}, W_{\nu}, f_{j}\rangle$ is homomesic and the average of $f_j$ along $W_{\nu}$ orbits is $k/m$.
\end{itemize}
\end{enumerate}
\end{theorem}

We will prove the above theorem in the next section.  Given the bijection in Corrollary \ref{bijection}, Theorem \ref{main} is a straightforward conclusion of Theorem \ref{winchthm}.
In addition, Theorem \ref{PRthm} can be concluded from the above theorem. In fact, a more general statement is shown in the next section (Corollary \ref{sizematters}).


The following variation of winching is called \textbf{winching with lower bounds} and it corresponds to comotion on $J({\cal U}_a)$.

\begin{defin}

Consider the sequence of lower bounds $l=(l_{1}, \dots, l_{k}) , ~ 0<l_{1}<\dots <l_{k}<m{+}1$ and $S'_{k,m}=\{(x_1, x_2,\dots x_k)\in S_{k,m}\vert  x_i\geq l_i\}$, where $ S_{k,m}$ is the set defined in Definition \ref{Winchdef}.
For any index $i\in [k]$, we define the map ${\underline{W}}_i: S'_{k,m}\rightarrow S'_{k,m}$  called {winching with lower bounds $l$ on index $i$} by $$\forall w\in S'_{k,m} \quad {\underline{W}}_{i}(w)=\max \{W_i(w),l_{i}\},$$where $W_i$ is the action of winching on index $i$ (Definition \ref{Winchdef}).
Having ${\cal U}_a$ be the poset which is defined in Definition \ref{posets}, we will have:

\end{defin}

\begin{lemma}\label{beta}

There is a bijection $\beta:J({\cal U}_a)\rightarrow S'_{a,2a}$ such that for the lower bounds $l=(1,3,5,\dots,2a{-}1)$, we have: for any $I\in J({\cal U}_a)$, $\beta(\sigma_{C_j}(I))= {\underline{W}}_{j}(\beta(I))$.
 
\end{lemma}

\begin{proof}
Fix arbitrary $a$ and consider $I\in J({\cal U}_a)$, we define $\beta(I)=(\beta_{1},\dots ,\beta_{a})$ as follows: for any $1\leq j\leq a$, $ \beta_{j}(I)=\vert  C_{a{+}1{-}j}(I)   \vert +2j{-}1$.
Considering  the ideal $I'\in J({\cal Q}_a),~ I'=I\cup ({\cal Q}_{a,a}-{\cal U}_a)$, we will have, $\beta(I)=\alpha(I')$. Hence, $\beta$ is an increasing sequence.  Since $\sigma_{C_j}(I)=\sigma_{C_j}(I')-({\cal Q}_{a,a}-{\cal U}_a)$ we have,

\begin{equation}
\beta_j(\sigma_{C_{a+j-1}}(I))= |\sigma_{C_{a+j-1}}(I)|+2j-1=|\sigma_{C_{a+j-1}}(I')-({\cal Q}_{a,a}-{\cal U}_a)|+2j-1
\end{equation}

\begin{equation}
\Rightarrow \beta_j(\sigma_{C_{a+j-1}}(I))= \max \{|\sigma_{C_{a+j-1}}(I')| -j+1,0 \}+2j-1=\max\{ \vert\sigma_{C_{a+j-1}}(I')\vert+j,2j-1\}
\end{equation}

\begin{equation}
\Rightarrow \beta_j(\sigma_{C_{a+j-1}}(I))=\max\{ (W_j(\alpha(I')))_j,2j-1\}=\max\{(W_j(\beta(I)))_j,2j-1\}.
\end{equation}
\end{proof}

\begin{con}\label{bijection2}
Consider any arbitrary permutation $\nu:[a]\rightarrow[a]$, the action ${\cal T}_{\nu}: J({\cal U}_a)\rightarrow J({\cal U}_a)$ and $I\in J({\cal U}_a)$. For any $x\in S_{a,2a}$, let the lower bounds be $l=(1,3,\dots,2a{-}1)$. Then: $\underline{W}_{\nu}(x)= \underline{W}_{\nu(a)}\circ \underline{W}_{\nu(a{-}1)}\circ\dots \circ \underline{W}_{\nu(1)}(x)$.
 Bijection $\beta$ introduced in \ref{beta} satisfies the following property: $\beta({\cal T_{\nu}}(I))=\underline{W}_{\nu} (\beta(I))$.
\end{con}

\begin{theorem}\label{winchlowthm}

Let $\nu$ be an arbitrary permutation of $[a]$. Consider $\underline{W}_{\nu}:S'_{a,b}\rightarrow S'_{a,b}$ with lower bounds $(l_1,l_2,\dots, l_a)$.
For each $i\in[a{+}b]$ let $[i,a{+}b]={i,i+1\dots a{+}b}$ and $f:[a{+}b]\rightarrow \mathbb{R}$ a function  that has the same average in all $[{l_i},a+b]$ , $1\leq i\leq a$.  Let $g:S'_{a,b}\rightarrow \mathbb{R}$ be defined as, $g(x)=\sum_{i=1}^a f(x_i)$. 
Then, the triple $\langle S'_{a,b}, \underline{W}_{\nu}, g\rangle$ exhibits homomesy.

\end{theorem}


We now define the action of \textbf{winching with zeros} to study homomesy in $J({\cal L}_a)$. 

\begin{defin}

Let $S_n$ be the set of all $n$-tuples $x=(x_{1},\dots,x_{n}) $ such that for some $0 \leq k \leq n$ $x_1=x_2\dots=x_k=0$ and $1 \leq x_{k{+}1}<x_{k{+}2}\dots<x_{n}\leq n$.
We define the map $\mbox{WZ}_i:S_n\rightarrow S_n$, called {winching with  zeros on  index $i$} to be

$$\mbox{{WZ}}_{i}(x)=  \left\{ \begin{array}{ll}

  x_{i}{+}1& \mbox{if  }  x_{i}{+}1<\min\{x_{i+1}, n{+}1\}    ; \\
 x_{i-1}{+}1& \mbox{if  } 1< i \mbox{ and } 0<x_{i-1}; \\
0 & \mbox{otherwise.} \\
\end{array} \right.
$$

 \end{defin}

 \begin{lemma}\label{gamma}\label{gamma}

There is a bijection $\gamma:J({\cal L}_a)\rightarrow S_{a}$ such that: for any $I\in J({\cal L}_a)$, $\gamma(\sigma_{C_j}(I))= {\mbox{WZ}}_{j}(\gamma(I))$.
 
\end{lemma}

\begin{proof} 
 Fix an arbitrary natural number $a$ and consider ${ I} \in J({\cal L}_a)$. We define, $\gamma (I)=(\gamma_1,\gamma_2, \dots,\gamma_a)$ as follows: for $1\leq j\leq a$, $\gamma_j(I)=\vert C_{a{+}1{-}j}(I)\vert$.
 For any $j_1<j_2$, we have $\vert C_{j_1}(I)\vert > \vert C_{j_2}(I)\vert $. Hence, $\gamma$ will be an increasing sequence.

Let $C_j=\{v_j, v_{j+1},\dots v_a\}$ where for $j\leq i\leq a$, $v_i= (j,i)$.  Assume $\vert C_j(I)\vert=l$, which means $C_j(I)=\{v_j, v_{j+1},\dots v_{j+l-1}\}$. For $n>j+l$, $\sigma_{v_n}(I)=I$. We have three cases: if $n=j+l$, we will have $\sigma_{v_n}(I)=I \cup \{v_n\}$ if and only $(j-1, j+l)\in I$ i.e. $\vert C_{j-1}(I)\vert> l+1$. If $C_{j+1}(I)=0$, $\sigma_{C_j}(I)= I- C_j(I)$. And if $\sigma_{C_j}(I)>0$, then $\sigma_{C_j}(I)=I-\{v_{k+1},\dots,v_{j+l-1}\}$, where $k=\vert C_{j+1}(I)\vert$.
Letting $\sigma_{C_j}(I)=K$, we will have:

\begin{equation}
 {C_{j}}(K) = \begin{cases}
C_j(I)\cup \{v_{j{+}l}\}, & \text{if } \vert C_{j{-}1}(I)\vert > l{+}1 .\\
\emptyset& \text{if } \vert C_{j{-}1}(I)\vert \leq l+1 \text{ and} \vert C_{j{+}1}(I)\vert=0 .\\
C_j(I)-\{v_{k+1},v_{k+2},\dots, v_{j+l-1}\}, & \text{otherwise}.\\
 \quad k=\vert C_{j{+}1}(I)\vert>0&
  \end{cases}
\end{equation}

\begin{equation}
\vert {C_{j}}(K) \vert= \begin{cases}
l+1, & \text{if } \vert C_{j{-}1}(I)\vert > l{+}1 .\\
0 & \text{if } \vert C_{j{-}1}(I)\vert \leq l+1 \text{ and} \vert C_{j{+}1}(I)\vert=0 .\\
k+1, & \text{otherwise}.\\
 \quad k=\vert C_{j{+}1}(I)\vert>0&
  \end{cases}
\end{equation}

\begin{equation}
{\gamma_{j}}(K) = \begin{cases}
\gamma_j (I)+1, & \text{if } \gamma_{j+1}(I)> l{+}1 .\\
0 & \text{if }    \gamma_{j+1}(I) \leq l+1 \text{ and} \gamma_{j-1}(I)=0 .\\
\gamma_{j-1}+1, & \text{otherwise}.\\

  \end{cases}
\end{equation}

\begin{equation}
\Leftrightarrow \gamma_{a{+}1{-}j}(\sigma_{C_{j}}(I))= \mbox{\emph {WZ}}_{a{+}1{-}j}( \gamma(I)).
\end{equation}
\end{proof}

 \begin{con}\label{bijection2}
Consider any arbitrary natural number $[n]$ and permutation $\nu$ on $n$, the action ${\cal T}_{\nu}: J({\cal L}_a)\rightarrow J({\cal L}_a)$, and $I\in J({\cal L}_a)$. For any $x\in S_{a}$,  we will have: $\mbox{WZ}_{\nu}(x)= \mbox{WZ}_{\nu(a)}\circ \mbox{WZ}_{\nu(a{-}1)}\circ\dots \circ \mbox{WZ}_{\nu(1)}(x)$.
 The bijection $\gamma$ introduced in \ref{gamma} satisfies the following property: $\gamma({\cal T_{\nu}}(I))=\mbox{WZ}_{\nu} (\gamma(I))$.
\end{con}

\begin{theorem}\label{zerowinchthm}
Consider an arbitrary natural number $n$ and an arbitrary permutation $\nu$ of $[n]$. With $\mbox{WZ}_{\nu}:S_{n}\rightarrow S_n$ defined as above we will have,

\begin{enumerate}

\item $\mbox{WZ}_{\nu}^{2n}(x)=x$ for all $x\in S_{n}$.

\item 

 For an arbitrary $1\leq  j\leq n$, let $f_j:S_{n}\rightarrow \mathbb{R}$ be defined by:

\begin{equation}
f_{j}(x)=\begin{cases}
1,& \mbox{if } j\in x\\
0,& \mbox{otherwise.}
\end{cases}
\end{equation}

The triple $\langle S_{n}, \mbox{WZ}_{\nu}, f_{j}\rangle$ is homomesic and the average of $f_j$ along $\mbox{WZ}_{\nu}$-orbits is $1/2$. Moreover, any linear combination of $f_j$s is homomesic in $\mbox{WZ}_{\nu}$-orbits of $S_n$.

\end{enumerate}
\end{theorem}

We will prove the above theorem in Section \ref{winch}.  Given the bijection in Corollary \ref{bijection}, Theorem \ref{zeromain} is a straightforward consequence  of Theorem \ref{zerowinchthm}.


\section{Some homomesy results in the comotion-orbits of $J({\cal Q}_{a,b})$, $J({\cal L}_{a})$, and $J({\cal U}_{a})$.}\label{results}

The following homomesy results can be easily verified using Theorem \ref{main}.

 \begin{con}\label{sizematters}
Let $\cal P$ be  ${\cal Q}_{a,b}$ or ${\cal L}_a$. Consider an arbitrary natural number $a$, an arbitrary  permutation $\nu$, and   ${\cal T}_{\nu}:J({\cal P})\rightarrow J({\cal P})$ as defined in \ref{Comotion}.
 We define the size function, $f:J({\cal P})\rightarrow\mathbb{R}$ as,  $\forall I, f(I)=\vert I\vert$. The triple
$\langle J({\cal P}), {\cal T}_{\nu}, f \rangle$ is homomesic for any choice of ${\nu}$ . 
 \end{con}
 
\begin{proof}
For ${\cal P}={\cal Q}_{a,b}$, $f=\sum_{i=1}^{a} i ~s_{i} - a(a+1)/2$. For  ${\cal P}={\cal L}_a$, $f=\sum_{i=1}^{a} i ~s_{i} $.
In both cases $f$ is a linear combination of $f_i$ using Theorems \ref{main} and \ref{zeromain} we will have the result.  
\end{proof}

\begin{con}
Consider the lattice ${\cal Q}_{a,b}$ and an arbitrary permutation $\nu$ of $[a]$.  Let $x\in[a]\times[b]$.  We define the antipodal
function $A:[a]\times[b]\rightarrow [a]\times[b]$ by $A(x)=y$ where $x=(i,j)\Leftrightarrow y=(a-i+1,b-j+1)$.
For $I\in J({\cal Q}_{a,b})$ and $x\in [a]\times [b]$, we define the characteristic function ${\cal I}_{ I}(x): [a]\times[b]\rightarrow \{0,1\}$s follows:

\begin{equation}
{\cal I}_{ I}(x)=\begin{cases}
1 & \mbox{if } x\in{\cal I}\\
0 & \mbox {otherwise} 
\end{cases}
\end{equation}

For any arbitrary  $x\in [a]\times [b]$ let $h: J({\cal Q}_{a,b}) \rightarrow \{0,1,-1\}$ be given by
$ h({ I})= {\cal I}_{ I}(x)-(1-{\cal I}_{I}(A(x)))$.  Then $h$ is $0-$mesic in ${\cal T}_{\nu}-$orbits of $J({\cal Q}_{a,b})$.
In other words, we have \textbf{central antisymmetry}, i.e. the average of number of ideals that contain $x$ is equal to the number
of ideals that do not contain $A(x)$.
\end{con}
 
 \begin{proof}
 Consider arbitrary ${I}\in{\cal Q}_{a,b}$ and $x=(x_{1},x_{2})\in [a]\times[b]$.  Then
 
 \begin{equation}\label{11}
 \begin{array}{ll}
 I_{{\cal I}}(x)=1 \Leftrightarrow (x_{1},x_{2}) \in {\cal I} \Leftrightarrow \vert C_{x_{1}}(I)\vert \geq x_{2} \\
 \Rightarrow I_{\cal I}(x)= \sum_{j=x_{2}}^{b} g_{x_1,j}.
 \end{array}
 \end{equation}
 
 Similarly, 
 
  \begin{equation}\label{12}
 \begin{array}{ll}
1- {\cal I}_{{I}}(A(x))=1 \Leftrightarrow (a-x_{1}+1,b-x_{2}+1) \notin { I} \Leftrightarrow \vert C_{a-x_{1}+1}(I)\vert <b- x_{2}+1 \Leftrightarrow \\ \vert C_{a-x_{1}+1}(I)\vert \leq b- x_{2}  
 \Rightarrow 1-{\cal I}_{ I}(A(x))= \sum_{j=0}^{b-x_{2}} g_{a-x_1+1,j}=  \sum_{j=x_{2}}^{b} g_{a-x_1+1,b-j}.
 \end{array}
 \end{equation} 
 
 By Equations \ref{11} and \ref{12}, we have $h_{x}({\cal I})= \sum_{j=x_{2}}^{b} g_{x_1,j}-g_{a-x_1+1,b-j}$. Employing Theorem \ref{main} we deduce that
$h_{x}$ is $0-$mesic for any arbitrary $x\in[a]\times [b]$.
 \end{proof}

\begin{con}
Let $\cal P$ be one of ${\cal Q}_a$ or ${\cal U}_a$. Consider arbitrary ${  I}\in J({\cal P })$. We denote the \textbf{rank-alternating} cardinality of $ I$ by ${\cal R}({ I})$ and we define it as
${\cal R}({ I})=\sum_{(i,j)\in{  I}} (-1)^{i{+j}}$. The triple $\langle J({\cal P}), {\cal T}_{\nu}, {\cal R}\rangle$ is homomesic for any arbitrary permutation $\nu$ of $[a]$.
\end{con} 

\begin{proof}
We will first consider the case when ${ I}\in J({\cal Q}_{a,b})$. In this case we have: 

\begin{equation}
\begin{array}{ll}
2~{\cal R}({ I})=  \sum_{x=(i,j)\in{\cal P}} (-1)^{i+j} {\cal I}_{I}(x)= \sum_{x=(i,j)} (-1)^{i+j}{\cal I}_{I}(x) + \sum_{x=(i,j)} (-1)^{i+j} {\cal I}_{I}(x)\\
\\
\Rightarrow2~{\cal R}({I})=   \sum_{x=(i,j)\in {\cal X}} (-1)^{i+j} {\cal I}_{I}(x) + (-1)^{2a-(i+j)+2} {\cal I}_{I}(A(x))\\ 
\\
=   \sum_{x=(i,j)\in {\cal X}} (-1)^{i+j} h(x)+1.
\\
\end{array}
\end{equation}

In the case where ${ I}\in J({\cal U}_a)$ we have:

$${\cal R}({ I})=(-1)^{a+1}\sum_{i ;  |C_i|\text{odd} }1. $$ 

We define the function $f: \mathbb{N}\rightarrow \{0,1\}$  as follows: $f(x)=1 \text{ iff } x \text{ odd}, f(x)=0 \text{ otherwise}. $
Note that the average of $f$ in any $[i,2a]$ that $i$ is odd is equal to $1/2$. Therefore, by Theorem \ref{mainU} we will have the result.  

\end{proof}


\section{Homomesy in winching}\label{winch} 

In this section we will prove Theorems \ref{winchthm}, \ref{winchlowthm}, and \ref{zerowinchthm}. The concepts of \textbf{tuple board} and \textbf{snake} are the prime definitions of this section, and they help us understand the orbit structure and homomesy in winching.

Fix $k$, for arbitrary $\nu$ a permutation of $[k]$, let ${\cal F}_{\nu}$ be one of $W_{\nu}$, $\underline{W}_{\nu}$ or $\mbox{\emph{WZ}}_{\nu}$. Let $S=S_{k}$ if ${\cal F}=\mbox{\emph{WZ}}_{\nu}$ and $S=S_{k,m}$ otherwise. 
We define a {tuple board}  as follows:

\begin{defin}\label{tupleA}

Consider  $x\in S$ and $\nu$ a permutation of $[k]$.
We write $x, {\cal F}_{\nu}(x), {\cal F}^{2}_{\nu}(x), \dots $ in separate, consecutive rows as depicted below.
Let $TB(x)=[x^1,x^2,\dots]$  be such a table, where $TB(i,\cdot )=x^i=(x^i_{1},\dots, x^i _{k})$ and $ x^{i}={\cal F}^{i{-}1}_{\nu}(x)$.
We will have a board looking as follows: 

\medskip
\vspace{5 mm}
\begin {tabular}{l|l|l|l|}
 \cline{4-4}\cline{3-3}\cline{2-2}
row 1 ($x^{1}$)&$x^{1}_{1}$&\dots&$x^{1}_{k}$\\ \cline{4-4}\cline{3-3}\cline{2-2}
row 2 ($x^2$)&$x^{2}_{1}$&\dots&$x^{2}_{k}$\\ \cline{4-4}\cline{3-3}\cline{2-2}
row 3 ($x^3$)& $x^{3}_{1}$&\dots&$x^{3}_{k}$\\
\dots&\dots &\dots &\dots\\
\dots&\dots &\dots &\dots
\end{tabular}

\medskip
{\footnotesize
\textbf{Figure 2.} A tuple board. 
}

\vspace{5 mm}

$TB(x)$ is called the {tuple board} of $x$.
Since ${\cal F}_{\nu}$ is a permutation, there is some $n$ such that ${{\cal F}_{\nu}}^{n+1}(x)=x$. Therefore, we can also define a cylinder corresponding to the orbit containing $x$:

Consider ${\cal O}$, an ${\cal F}_{\nu}$-orbit of  $S$ which is produced by applying ${\cal F}_{\nu}$ consecutively to $x$.
We define the \textbf{tuple cylinder} $TS({\cal O})$ to be the cylinder that is produced by attaching the first and the $n{+}1$st row of $TB(x)$. 
Since $\cal O$ is an orbit it is more natural to think of a tuple board as a cylinder. We will use the terms interchangeably in this text.

\end{defin}

Notice that  any cell in a  tuple board contains a number from the set $\{0,1,2,\dots , m\}$. In what comes in the following we will introduce the notion of snakes. Given a tuple board $T$, any snake in it,  is a sequence of adjacent cells in $T$ that contain the numbers $1,2,\dots m$. The mathematical definition of a snake comes in the following:


\begin{defin}\label{snak}
For arbitrary $\nu=(\nu_{1},\nu_{2},\dots,\nu_{k})$ a permutation of $[k]$ and $x\in S$, let $TB=TB(x)$ be the tuple board of $x$ as defined in Definition \ref{tupleA}.
Considering $T=\{TB(i,j)| 1\leq i\leq n,1\leq j\leq k \}$, we
define a  \textbf{snake} $s=(s_{f}, s_{f+1}, \dots , s_{t})$ as follows: $s$ is a maximal sequence of $s_i$s such that each $s_{i}$ is a cell in the tuple board containing $i$,
and for $i>f$, $s_{i}={\cal M} (s_{i-1})$, where $\cal M$ is defined as follows:
\begin{equation}
{\cal M}(T(i,j))= 
\begin{cases}
T(i{+}1,j) & \mbox{if } T(i{+}1,j)=T(i,j)+1.\\
T(i,j{+}1) & \mbox{if } T(i,j{+}1)=T(i,j)+1  ,\quad T(i{+}1,j)\neq T(i,j)+1 \\ 
&\mbox{ and } \nu(j)<\nu(j+1). \\
 T(i{+}1,j{+}1) & \mbox{if } T(i{+}1,j{+}1)=T(i,j)+1 ,   T(i{+}1,j)\neq T(i,j)+1 \\
 & \mbox{ and } \nu(j)>\nu(j+1). \\ 

\end{cases}
\end{equation}

\end{defin}

\begin{defin}

Consider $T=TB=[x^1,x^2,\dots,x^n]$ as defined previously for $x\in S$.
In what follows row numbers in a tuple board are understood modulo $n$.

 Consider $s$  a snake in $T$. We define  \textbf{snake map}  $\cal S$, a function that associates any snake with an element in $\mathbb{N}^{k}$ as follows: for an arbitrary snake $s$, ${\cal S}(s)= (c_{1},c_{2},\dots,c_{k})$, where $c_{j}=\vert \{i|T(i,j)\in s\}\vert$. 

\end{defin}

\medskip

\begin {tabular}{|l|l|l|l|l|l|}
\dots &\dots &\dots&\dots &\dots\\
 \cline{4-4}\cline{3-3}\cline{2-2}\cline{5-5}\cline{1-1}
$1$&$2$&$?$&$?$&$?$\\ \cline{4-4}\cline{3-3}\cline{2-2}\cline{5-5}\cline{1-1}
$?$&$3$&$?$&$?$&$?$\\ \cline{4-4}\cline{3-3}\cline{2-2}\cline{5-5}\cline{1-1}
$?$& $4$&$5$&$?$&$?$\\\cline{4-4}\cline{3-3}\cline{2-2}\cline{5-5}\cline{1-1}
$?$ &$?$ &$6$&$?$ &$?$\\\cline{4-4}\cline{3-3}\cline{2-2}\cline{5-5}\cline{1-1}
$?$ &$?$ &$7$&$?$ &$?$\\\cline{4-4}\cline{3-3}\cline{2-2}\cline{5-5}\cline{1-1}
$?$& $?$&$?$&8&$9$\\\cline{4-4}\cline{3-3}\cline{2-2}\cline{5-5}\cline{1-1}
$?$& $?$&$?$&$?$&$10$\\\cline{4-4}\cline{3-3}\cline{2-2}\cline{5-5}\cline{1-1}
\dots &\dots &\dots&\dots &\dots\\
\end{tabular}
\quad \quad 
\begin {tabular}{|l|l|l|l|l|l|}

\dots &\dots &\dots&\dots &\dots\\
 \cline{4-4}\cline{3-3}\cline{2-2}\cline{5-5}\cline{1-1}
$?$&$4$&$?$&$?$&$?$\\ \cline{4-4}\cline{3-3}\cline{2-2}\cline{5-5}\cline{1-1}
$?$& $?$&$5$&$?$&$?$\\\cline{4-4}\cline{3-3}\cline{2-2}\cline{5-5}\cline{1-1}
$?$ &$?$ &$6$&$?$ &$?$\\\cline{4-4}\cline{3-3}\cline{2-2}\cline{5-5}\cline{1-1}
$?$ &$?$ &$7$&$?$ &$?$\\\cline{4-4}\cline{3-3}\cline{2-2}\cline{5-5}\cline{1-1}
$?$& $?$&$8$&9&$10$\\\cline{4-4}\cline{3-3}\cline{2-2}\cline{5-5}\cline{1-1}
$?$& $?$&$?$&$?$&$?$\\\cline{4-4}\cline{3-3}\cline{2-2}\cline{5-5}\cline{1-1}
\dots &\dots &\dots&\dots &\dots\\
\dots &\dots &\dots&\dots &\dots\\
\end{tabular}
\quad \quad 
\begin {tabular}{|l|l|l|l|l|l|}

\dots &\dots &\dots&\dots &\dots\\
\dots &\dots &\dots&\dots &\dots\\
 \cline{4-4}\cline{3-3}\cline{2-2}\cline{5-5}\cline{1-1}
 $0$&$1$&$?$&$?$&$?$\\ \cline{4-4}\cline{3-3}\cline{2-2}\cline{5-5}\cline{1-1}
$?$&$?$&$2$&$?$&$?$\\ \cline{4-4}\cline{3-3}\cline{2-2}\cline{5-5}\cline{1-1}
$?$&$?$&$3$&$4$&$5$\\ \cline{4-4}\cline{3-3}\cline{2-2}\cline{5-5}\cline{1-1}
$?$& $?$&$?$&$?$&$?$\\\cline{4-4}\cline{3-3}\cline{2-2}\cline{5-5}\cline{1-1}
$?$ &$?$ &$?$&$?$ &$?$\\\cline{4-4}\cline{3-3}\cline{2-2}\cline{5-5}\cline{1-1}
\dots &\dots &\dots&\dots &\dots\\
\dots &\dots &\dots&\dots &\dots\\

\end{tabular}

\vspace{5 mm}
{\footnotesize

\noindent
\textbf{Figure. 3.}\\

\noindent
\begin{tabular}{lll}
 A tuple board corresponding to $W_{\nu}$&  A tuple board corresponding to $W_{\nu}$ & A tuple board corresponding\\  
$x\in S_{5,10}$ and   &    with lower bounds $(2,4,6,7,8)$. & to ${\mbox{\emph WZ}}_{\nu}$. $x\in S_5$  and   \\
$\nu=(1,2,4,3,5)$. &  $x\in S_{5,10}$ and $\nu=(1,3,2,4,5)$. & $\nu=(1,3,2,4,5)$ \\
The snake map is $(1,3,3,1,2)$.& The snake map is $(0,1,4,1,1)$. & The snake map is $(0,1,2,1,1)$.
\end{tabular}

\medskip

}

\subsection{Proof of Theorem \ref{winchthm}}

In this subsection we prove Theorem \ref{winchthm}.

\begin{defin}\label{inverse}
Let $\bar{W}_{i}:S_{k,m}\rightarrow S_{k,m}$ be the following map: $\forall x=(x_{1}\dots x_{k})\in S_{k,m},~\bar{W}_{i}(x)=y=(y_{1},y_{2},\dots ,y_{k})$
where $\forall j\neq i, y_{j}=x_{j}$, and
\begin{equation}
 y_{i} = \begin{cases}
 x_{i+1}{-}1, & \text{if }  x_{i}=x_{i-1}+1 .\\
x_{i}{-}1, & \text{otherwise}.
  \end{cases}
  \end{equation}
Note that $\forall x\in S_{k,m}, \bar{W}_{i}\circ W_{i}(x)=x$. We call $\bar{W}_i$ \textbf{inverse winching} at index $i$.  
\end{defin}

\begin{defin}
For $\nu$ an arbitrary  permutation of ${[k]}$, $\bar{W}_{\nu}:S^{k,m}\rightarrow S^{k,m}$ is defined by $\bar{W}_{\nu}=\bar{W}_{\nu(1)}\circ \bar{W}_{\nu(2)}\circ\dots\circ W^r_{\nu(k)}$ and we have 
$\forall x\in S_{k,m}, \bar{W}_{\nu}(W_{\nu}(x))=x$. 
\end{defin}

\begin{lemma} \label{Wsnake}
Any snake in a tuple cylinder $TS({\cal O)}$ (where ${\cal O}$ is a $W_{\nu}$-orbit of $S_{k,m}$) is of length $m$, starts in the first column
of the cylinder with  $s_{1}$, and ends in the last column of the tuple cylinder with $s_m$.
\end{lemma}
\begin{proof}
Consider  some $x\in {\cal O}$ and a snake $s$ in  the tuple board $T=TB(x)=[x^{1},\dots , x^n]$. We assume that $s=(s_{f},\dots , s_{t})$. Having, $x^{i+1}=W_{\nu}(x^{i})$, it is easy to verify that unless $t=m$, we can find a cell in $T$ to expand $s$. 
Similarly, since $x^{i-1}=\bar{W}(x^{i})$. If $\nu(j)< \nu(j+1)$, we can see: unless $f=1$, the snake $s$ can be expanded.

\end{proof}

\begin{defin}
Let $H:[m]^{[k]}\rightarrow [m]^{[k]}$ be defined as follows: $\forall x=(x_{1},\dots, x_{k}),~H(x)=y=(y_{1},\dots, y_{k})$ where
$\forall 1\leq i< k,  y_{i}=x_{i+1} $ and  $y_{k}=x_{1}$. We call $H$ the \textbf{left shift operator}. 

\end{defin}

\begin{lemma}
\label{pp}
Let  $(p,1)$ and $(q,1)$ $(p<q)$ be two cells of tuple board $T$ with value 1, such that there is no $p<i<q$ with $T(i,1)=1$.
Consider the snake $s^{p}=(s^p_{1}\dots s^{p}_{m})$ starting with $s^p_1=T(p,1)$ and ${\cal S}(s^p)=c^{p}=(c^p_{1},\dots,c^p_{k})$
its snake map; and similarly consider the snake $s^q$ and its snake map ${\cal S}(s^q)=c^q$ starting at $T(q,1)$.  Then,
\begin{itemize}
\item If $T(i,j)\in s^{p}$, we have the following:
\begin{itemize}
\item$T(i+1,j)\notin s^{p} \Rightarrow T(i+1,j)\in s^q$. 
\item If  $j>1$ then, $T(i,j{-}1)\notin s^{p} \Rightarrow T(i,j{-}1)\in s^q$.
\end{itemize}
In other words there is no gap between two consecutive snakes in the tuple board. 
 \item We have $c^q=H(c^p)$. 
\end{itemize}

\end{lemma}

\begin{proof}
In order to prove this lemma we fix $\nu= (1,2,\dots, k)$. The proof will be similar for any arbitrary permutation $\nu$.
To make notation simpler we drop the subscript from $W$ meaning $\nu= (1,2,\dots, k)$.

Suppose that we have the action of winching $W_{(1,2,\dots, k)}$ on $x\in S_{k,m}$ making the orbit
${\cal O}$ in $ S_{k,m}$. Moreover, suppose the tuple board corresponding to $x$ (equivalently, the tuple cylinder corresponding to $\cal O$)
is $T=TB(x)=[x^1,x^2,\dots, x^n]$ where $x^{i}=W^{i-1}(x)$ is as defined in Definition \ref{tupleA}.

\textbf{Claim 1.}  $c^p_{1}=q-p$.

Since $T(q,1)\notin s^{p}$, $c^p_{1}\leq q-p$. Moreover $c^p_{1}=c_{1} \leq q-p$ implies $s^p_{c_1+1}=T(p{+}c_1{-}1,2)=c_1+1$  meaning  $x^{c_{1}}_{2}=c_1+1$ and $x^{c_{1}}_{1}=c_1$. 
We have  $x^{c_1+1}=W(x^c_1)$, and hence $x^{c_1+1}_{1}=1$, and $T(c_1+p,1)\in s^{q}\Rightarrow c_1+p=q\Rightarrow c_1=q-p$.

Note that Claim 1 implies that there is no gap between the two snakes in column 1. (See Figure 4).

\textbf{Claim 2.} $c^q_{1}=c^p_{2}$.
For simplicity, we denote $c^{p}_{1}$ by $c_{1}$ and $c^{p}_{2}$ by $c_2$. 

We have $s^p_{1}= T(p,1)=1, s^{p}_{2}=T(p+1,1), \dots ,s^{p}_{c_{1}}=T(p+c_{1}-1,1)$.  (See Figure 4)

Then, for all $1\leq i\leq c_{2}$:

\begin{equation}\label{eq2}
\begin{array}{llll}
&s^{p}_{c_{1}+i}&=T(p+c_{1}{-}1{+}(i{-}1),2) \\
\Rightarrow   &s^{p}_{c_{1}+i}&=T(q{+}(i{-}2),2)&  (\mbox{Since } q=p+c_{1})\\
 \Rightarrow & x^{c_{1}+i-1}_{2}&=c_{1}+i
\end{array}
\end{equation}

We also have

\begin{equation}\label{eq1} s^{p}_{c_{1}+c_{2}+1}=T(q+c_{2}-2,3)=T(p+c_{1}+c_{2}-2,3) \Rightarrow x^{c_{1}+c_{2}-1}_{3}=c_{1}+c_{2}+1. 
\end{equation}

Now consider $s^q$. 
   For all $i, 1\leq i\leq c_{2}-1$ we have that if $x^{c_{1}+i}_{1}=i$,
 
\begin{equation}\label{eq3}  \left.
 \begin{array}{l}
  x^{c_{1}+i}_{1}=i \\
   x^{c_{1}+i}_{2}=c_{1}+i+1>i  \\
 W(x^{c_{1}+i})=x^{c_{1}+i+1}
 \end{array}
 \right\}
 \Rightarrow x^{c_{1}+i+1}_{2}=i+1\end{equation}
Therefore,  
\begin{equation}
x_{1}^{c_1+1}=1  \Rightarrow   \forall i, 1\leq i\leq c_{2}-1,~{\cal M}(s^q_{i})=T(q+i,1) 
\Rightarrow
  \forall 1\leq i\leq c_{2},~s^q_{i}=T(q+i-1,1).
\end{equation}

From Equation \ref{eq3} we can conclude $x^{c_{1}+c_{2}}_{1}=c_{2}$. By Equations \ref{eq1} and \ref{eq2}, and the fact that
$W(x^{c_{1}+c_{2}-1})=x^{c_{1}+ c_{2}}$, we have $x^{c_{1}+c_{2}}_{2}=c_{2}+1$. Hence, ${\cal M}(s^{q}_{c_{2}})=T(q+c_{2}-1,2).$

It follows that $c^q_{1}=c^p_{2}$. Moreover,  $ T(i,2)\in s^{p}\Rightarrow T(i,1)\in s^q$ and $T(i-1,2)\in s^{q}$  for any $i$ (if they are
not already in $s^p$).

Very similar to the proof of Claim 2, the following can be proved using the definitions:

\textbf{Claim 3.} Let $r < k$ with $\forall l, 1\leq l<r-1, ~c^q_{l}=c^p_{l+1}$, then $c^q_r=c^p_{r+1}$.









Having Claims 2 and 3, by employing induction we can show: for all $ i, 1\leq i<k-1, c^q_i=c^p_{i+1}$, and that there is no gap between the snakes in any of the columns. Furthermore, since all snakes have the same length, $c^q_{k}=c^p_{1}$.$~\Box$

\hbox{}

\textbf{Proof of Theorem \ref{winchthm}, Part 1.}
Consider an $n \times k$ tuple board $T$ such that $T=TB(x)$ and $x\in S_{k,m}$.
Let's assume that $n\geq m$ (if $n<m$, append enough copies of $T$ to it until $n\geq m$). Let  $s^{1}$ be the snake that covers $T(1,1)$, $s^{2}$ the next snake immediately below $s^{1}$, and $s^{i}$
the last snake right below $s^{i-1}$.  Letting ${\cal S}(s^{1})=c=(c_{1},c_{2},\dots,c_{k})$, we have ${\cal S}(s^{i})=H^{i-1}(c)$.
The numbers in the first column of $T$ will be: $x_{1}, x_{1}+1,\dots, x_{1}{+}c_{1}{-}1, 1, 2 ,\dots c_{2},1, 2, \dots, c_{3},\dots$.
Since $\sum_{i=1}^{k}c_{i}=m $,  the $m+1$st number in the first column will be $x_{1}$. 
 Similarly, for each column $i$, the $m{+}1$st element will be $x_{i}$. Thus, $W^{m+1}(x)=x$. 

\end{proof}

 \begin{con}

 The above reasoning also shows there are exactly $k$ snakes covering an $m\times k$ tuple cylinder. 
 \end{con}

\begin{con}
Fix $k$ and $n$ and $\nu$ a permutation of $[k]$. To each tuple cylinder $T$ of size $k\times n$ corresponding to a $W_{\nu}-$orbit,  we can assign a sequence $c=(c_1,c_2,\dots,c_k)$, satisfying $\sum_{i=1}^{k}c_i=n$ where $T$ is covered by snakes $s_1, s_2, \dots ,s_k$  and for all $1\leq j\leq k$,
there is an $i$ such that ${\cal S}(s_{j})=H^i(c)$. Since filling the first column of the cylinder will impose the other numbers, this correspondence is a one to one mapping. 
\end{con}

\begin {tabular}{l|l|l|l|l|l|l|l|}
  \cline{7-7}\cline{6-6}\cline{5-5} \cline{4-4}\cline{3-3}\cline{2-2}  \cline{8-8}
&&&&&\hspace{2cm}&\quad&\quad\\&&&&&&&\\&&&&&&&\\

  \cline{7-7}\cline{6-6}\cline{5-5} \cline{4-4}\cline{3-3}\cline{2-2} \cline{7-7}\cline{6-6}\cline{5-5} \cline{8-8}
&&&\dots&&&&\\ \cline{4-4}\cline{3-3}\cline{2-2} \cline{7-7}\cline{6-6}\cline{5-5} \cline{8-8}
row $p$ ($x^1$)&\color{red}$1$&&\dots&&&&\\ \cline{4-4}\cline{3-3}\cline{2-2} \cline{7-7}\cline{6-6}\cline{5-5} \cline{8-8}
row $p{+}1$ ($x^2$)& \color{red}$2$&&\dots&&&&\\
  \cline{7-7}\cline{6-6}\cline{5-5} \cline{4-4}\cline{3-3}\cline{2-2} \cline{8-8}
&&&&&\hspace{1cm}&\quad&\quad\\&&&&&&&\\&&&&&&&\\
 \cline{4-4}\cline{3-3}\cline{2-2} \cline{7-7}\cline{6-6}\cline{5-5} \cline{8-8}
row $p{+}c_1{-}1$ ($x^{c_1}$)&\color{red}$c_1$&\color{red}$c_1{+}1$&\dots&&&&\\ \cline{4-4}\cline{3-3}\cline{2-2} \cline{7-7}\cline{6-6}\cline{5-5} \cline{8-8}
row $q$ ($x^{c_1+1}$)&\color{blue}$1$&$\color{red}c_1{+}2$&\dots&&&&\\ \cline{4-4}\cline{3-3}\cline{2-2} \cline{7-7}\cline{6-6}\cline{5-5} \cline{8-8}
&&&&&\hspace{1cm}&\quad&\quad\\&&&&&&&\\&&&&&&&\\
 \cline{4-4}\cline{3-3}\cline{2-2} \cline{7-7}\cline{6-6}\cline{5-5} \cline{8-8}
row $q{+}c_2{-}1$ ($x^{c_1+c_2-1}$)&\color{blue}$c_2$&\color{blue}$c_2{+}1$&&&&&\\
 \cline{4-4}\cline{3-3}\cline{2-2} \cline{7-7}\cline{6-6}\cline{5-5} \cline{8-8}
row $q{+}c_2$ ($x^{c_1+c_2}$)&\color{green}$1$&\color{blue}$c_2{+}2$&\dots&&&&\\ \cline{4-4}\cline{3-3}\cline{2-2}  \cline{7-7}\cline{6-6}\cline{5-5} \cline{8-8}

&&&&&\hspace{1cm}&\quad&\quad\\&&&&&&&\\&&&&&&&\\
 \cline{4-4}\cline{3-3}\cline{2-2} \cline{7-7}\cline{6-6}\cline{5-5} \cline{8-8}

\end{tabular}
\vspace{0.5cm}

{\footnotesize
\emph{Figure. 4. Snakes in a tuple board of $W_{(1,2,\dots, a)}$.}
}

\vspace{0.5cm}

In order to prove Part 2 of Theorem \ref{winchthm}, we present the definition of rotation on  2-colored necklaces with $k$ white beads and $n-k$ black beads, then we proceed to the proof: 

\begin{defin}\label{Neck}

Let $N_{k,m}$ the set of all $k$-tuples $(x_1, x_2 , \dots , x_k)$  satisfying $1\leq x_1 <x_2 < \dots < x_k\leq m$ and $\sum_{i=1}^k x_i =m$.
The action of rotation on this set is defined as $R:N_{k,m}\rightarrow N_{k,m}$, $\forall x\in N_{k,m},~ R(x)=y$, where $y=(x_1{+}1,x_2{+}1,\dots ,x_k{+}1)$ if for all $i, x_i< m$. Or  $y=(1,x_1{+}1,\dots x_{k-1}{+}1)$ if $x_k=m$.

\end{defin}


\begin{lemma}\label{map1}

There is a map ${\cal K}:S_{k,m}\rightarrow N_{k,m}$ satisfying $\forall x\in S_{k,m},~R({\cal K}(x))= {\cal K}(\bar{W}_{\nu}(x))$.

\end{lemma}

\begin{proof}
Consider arbitrary $x\in S_{k,m}$ and $T=TB(x)$ as in Definition \ref{tupleA}. Let $s$ be the snake covering $T(1,1)$.
For ${\cal S}(s)=(c_{1},c_{2},\dots ,c_{k})$,  we define
${\cal K}(x)=(y_1, y_2 \dots y_k)\in N_{k,m}$, where $ y_{1}=c_{1}{-}x_{1}{+}1$, and for $2\leq i\leq k$, $y_{i}=y_{i-1}+c_{i}$.

Note that  $R({\cal K}(x))= {\cal K}(\bar{W}(x))$ if and only if  $R({\cal K}(W(x)))= {\cal K}(x)$. Let $T=TB(x)$, $T'=TB(W(x))$. Let
$s$ be the snake in $T$  covering $T(1,1)$,  $c={\cal S}(s)$, and  similarly let $s'$ be the snake in $T'$ covering $T'(1,1)$,  $c'={\cal S}(s')$, and $W(x)=z$.
Either $c=c' $ and $ x_{1}+1=z_{1}$ or $c'=H(c)$,  $x_{1}=c_{1}$, and $z_1=1$.

\begin{equation}
R({\cal K} (z))=
R(y_1,y_2,\dots y_k) ;  y_{1}= c'_{1}-z_{1}+1, y_{i+1}=y_{i}+c'_{i+1}
\end{equation}
\begin{equation}=
\left\{
\begin{array}{ll}
R(y_1,y_2,\dots y_k) ;  y_{1}= c_{1}-(x_{1}+1)+1, y_{i+1}=y_{i}+c_{i+1}\\
R(y_1,y_2,\dots y_k) ;  y_{i}| y_{1}= c_{2}-1+1, y_{i+1}=y_{i}+c_{i+2}
\end{array}\right.
\end{equation}
\begin{equation}=
\left\{
\begin{array}{ll}
R(y_1,y_2,\dots y_k) ;  y_{1}= c_{1}-x_{1}, y_{i+1}=y_{i}+c_{i+1}\\
R(y_1,y_2,\dots y_k) ; y_{1}= c_{2}, y_{i+1}=y_{i}+c_{i+2(mod ~ k)}
\end{array}\right.
\end{equation}

\begin{equation}=
\left\{
\begin{array}{ll}
\{(y'_1,y'_2,\dots y'_k) ;  y_{1}= c_{1}{-}x_{1}{+}1, y_{i+1}=y_{i}{+}c_{i+1}\}) \quad \quad \quad \mbox {because } \forall i, y_{i}<m\\
R( \{y_{i}| y_{i}= c_{i+1}(1\leq i\leq k-1), y_{k}=m\}) =(1, 1{+}c_{2}, 1{+}c_{2}{+}c_{3}\dots 1+\sum_{i=1}^{k}c_{k})
\end{array}\right.
\end{equation}
\begin{equation}
={\cal K}(x)
\end{equation}

\end{proof}

\exm{Consider $x=(2,3,4,6)\in S_{4,7}$, $W(x)=(1,2,5,6)$. In Figure 5, $TB(x)$ and $TB(W(x))$ are depicted. We see that ${\cal K}(W(x))=(1,4,5,7)$, and ${\cal K}(x)=(1,2,5,6)$. Note that $R(1,4,5,7)=(1,2,5,6)$. 

 }

\vspace{0.5cm}
$T=TB(2,3,4,6):$
\begin {tabular}{l|l|l|l|l|l|l|l|}
 \cline{4-4}\cline{3-3}\cline{2-2}  \cline{5-5}
&2&3&4&6\\
\cline{5-5} \cline{4-4}\cline{3-3}\cline{2-2}
&1&2&5&7\\
\cline{5-5} \cline{4-4}\cline{3-3}\cline{2-2}
&1&3&6&7\\
\cline{5-5} \cline{4-4}\cline{3-3}\cline{2-2}
&2&4&5&6\\
\cline{5-5} \cline{4-4}\cline{3-3}\cline{2-2}
&3&4&5&7\\
\cline{5-5} \cline{4-4}\cline{3-3}\cline{2-2}
&1&2&6&7\\
\cline{5-5} \cline{4-4}\cline{3-3}\cline{2-2}
&1&3&4&5\\
\cline{5-5} \cline{4-4}\cline{3-3}\cline{2-2}
\end{tabular}
$\hspace{3cm} T'=TB(1,2,5,7):$
\begin {tabular}{l|l|l|l|l|l|l|l|}
\cline{5-5} \cline{4-4}\cline{3-3}\cline{2-2}
&1&2&5&7\\
\cline{5-5} \cline{4-4}\cline{3-3}\cline{2-2}
&1&3&6&7\\
\cline{5-5} \cline{4-4}\cline{3-3}\cline{2-2}
&2&4&5&6\\
\cline{5-5} \cline{4-4}\cline{3-3}\cline{2-2}
&3&4&5&7\\
\cline{5-5} \cline{4-4}\cline{3-3}\cline{2-2}
&1&2&6&7\\
\cline{5-5} \cline{4-4}\cline{3-3}\cline{2-2}
&1&3&4&5\\
\cline{5-5} \cline{4-4}\cline{3-3}\cline{2-2}
&2&3&4&6\\
 \cline{4-4}\cline{3-3}\cline{2-2}  \cline{5-5}
\end{tabular}

\vspace{0.5cm}

{\footnotesize
\emph{Figure. 5. The snake covering $T(1,1)$ has snake map (2,1,3,1). Hence, {\cal K}(x)=(1,2,5,6). The snake covering $T'(1,1)$ has snake map (1,3,1,2). Hence, {\cal K}(W(x))=(1,4,5,7).}
}

\vspace{0.5cm}

\textbf{Proof of Theorem  \ref{winchthm}, Part 2.}
Having Lemma \ref{map1}, we conclude that the orbit structures of $\langle N_{k,m}, R\rangle$ and $\langle S_{k,m},W \rangle $  are the same.

\begin{lemma}\label{map}
{
Let $T$ be an $m \times k$ tuple board.  Consider the column $r$: $T_{r}=\{T(i,r)\}$. For any $ 1\leq r \leq k $, there exists a one-to-one function ${\cal F}: T_{r}\rightarrow T_{k{+}1{-}r}$, satisfying ${\cal F}(x)= m{+}1{-}x$.
}
\end{lemma}
\begin{proof}
For any $r$, we construct a mapping from $\{\cup_{t=1}^{r}T_{t}\}$ to $\{ \cup_{t=k{-}r{+}1}^{k}T_{t} \}$. Consider a number $x$ in $T_{r}$.
Let it be the $l$th element in $T_{r}$, covered by a snake having snake map $p=(c_{1}, c_{2}, \dots,c_{r},\dots , c_{k}).$
Consider the snake with snake map $p'=(c_{r{+}1},\dots ,c_{1}, c_{2}, \dots,c_{r} )$. Let $y$ be the $\sum_{i=1}^{r{-}1}c_{i}{+}l$th element
 from the end in this snake. Then $y=m{+}1-\sum_{i=1}^{r{-}1}c_{i}{+}l= m{+}1{-}x$. Since $\sum_{i=1}^{r{-}1}c_{i}{+}l\leq \sum_{i=1}^{r}c_{i}$,
$y$ will be lying in one of the columns $k,\dots k{-}r{+}1$. 

Having the above mapping, we know there is also a one-to-one mapping in $\{\cup_{t=1}^{r}T_{t}\}\rightarrow\{ \cup_{t=k{-}r{+}1}^{k}T_{t} \}$
and also in $\{\cup_{t=1}^{r-1}T_{t}\}\rightarrow\{ \cup_{t=k{-}r}^{k}T_{t} \}$. Hence, there exists ${\cal F}: T_{r}\rightarrow T_{k{+}1{-}r}$
satisfying the lemma's conditions.

\end{proof}

\textbf{Proof of Theorem \ref{winchthm}, part 3.}
Considering any $m\times k$ tuple cylinder $TS({\cal O})$, Lemma \ref{pp} shows that $TS({\cal O})$ is totally covered by $k$ snakes.
Therefore, each element $1\leq i\leq m$ appears $k$ times in the cylinder and therefore the average of $f_i$ as defined in Theorem \ref{winchthm} part 3
is independent of $\cal O$ and equal to $k/m$.

Lemma \ref{map} shows that the number of $j$s in any column $i$ is equal to the number of $m{-}j{+}1$'s in column $k-i+1$ of $TS({\cal O})$.
Thus, $\sum_{x\in{\cal O}} g_{i,j}(x)=  \sum_{x\in{\cal O}}g_{k-i+1,m-j+1}(x)$.
In other words, $\forall 1\leq i\leq k, 1\leq j\leq m, g_{i,j}-g_{k-i+1,m-j+1}$ is $0$-mesic in $W$-orbits of $S_{k,m}$.
$\Box$

\subsection{Proof of Theorem \ref{winchlowthm}}

In this subsection we will prove Theorem \ref{winchlowthm}. Remember the definitions of tuple board, snake, snake map and the correspondence to the action of winching with lower bounds.

\begin{defin}
For the set of lower bounds $l=(l_1,\dots l_k)$ and $i\in k$, let $\underline{\bar{W}}_i$ be defined as: $\underline{\bar{W}}_i: S_{k,m}\rightarrow S_{k,m}; ~\forall x\in S_{k,m}  ~ \underline{\bar{W}}_i(x)=\max \{l_i,  \bar{W}_i (x)\}$,  where $\bar{W}_i $ is  defined as in Definition \ref{inverse}.

Note that $\forall x\in S_{k,m}, \bar{W}_{i}\circ W_{i}(x)=x$. We call $\bar{W}_i$ \textbf{inverse winching} at index $i$.  
\end{defin}

\begin{defin}
For arbitrary $k$ and $\nu$ a permutation of $[k]$, $\bar{W}_{\nu}:S_{k,m}\rightarrow S_{k,m}$ is defined by $\bar{W}_{\nu}=\bar{W}_{\nu(1)}\circ \bar{W}_{\nu(2)}\circ\dots\circ W^r_{\nu(k)}$ and we have 
$\forall x\in S_{k,m}, \bar{W}_{\nu}(W_{\nu}(x))=x$. 
\end{defin}

In contrast to what we showed in Lemma \ref{Wsnake} for $W_{\nu}$-snakes , $\underline{W}_{\nu}$-snakes do not necessarily cover all the numbers $1,2,\dots , m$. As depicted in Figure 3, they only contain $l_i,\dots m$ where $l_i$ is one of the lower bounds. The following lemma states this formally: 

\begin{lemma}

Consider a  tuple cylinder $TS({\cal O)}$ which is constructed by applying the action of $\underline{W}_{\nu}$ to an arbitrary $x\in S_{k,m}$. Let the lower bounds for this action be $l=(l_1,l_2,\dots ,l_k)$.
Any snake in this tuple board starts in some column $q$ and 
with  $s_{l_q}$, and ends in the last column of the tuple cylinder with $s_m$.
\end{lemma}

\begin{proof}
Consider  $x\in {\cal O}$ and  snake $s$ in  the tuple board $T=TB(x)=[x^{1},\dots , x^n]$. We assume that $s=(s_{f}\dots s_{t})$.  Following the definitions of winching and snakes, we can see that unless $t=m$ we can append more cells to the tail of the snake, and if $l_i<f$ we can append more cells to the head of the snake. 

\end{proof}

\noindent
\textbf{Proof of Theorem  \ref{winchlowthm}.}

Any tuple cylinder corresponding to action of $\underline{W}_{\nu}$ can be partitioned to snakes that start with some $s_{l_q}$ and end in $s_m$. Therefore, if $f$ is a function that have the same average on all the numbers contained in any snake, it will have the same average over all the elements in the tuple cylinders. Therefore, we will have the result. 

$\hspace{5.9 in}\Box$

\subsection{Proof of Theorem \ref{zerowinchthm}}

In this section, we will prove Theorem \ref{zerowinchthm}. Remember the definitions of tuple board, tuple cylinder, snake, snake map. 
Consider $x\in S_n$ and the action of $\mbox{\emph{WZ}}_{\nu}$ for some arbitrary permutation $\nu$ of $[n]$.
Note that in the case of winching with zeros, since we might have a bunch of zeros in our tuple board the snake does not necessarily start in column 1. However, it is a consecutive collection of numbers $1,2,\dots n$ as it was the case in Lemma \ref{Wsnake}. 

\begin{defin}
Let $M_{n}$ be the set of all sequences $(c_{1},\dots,c_n)$ that have $k$ preceding $0$s for some $0\leq k\leq n-1$,  $c_{k+1}\dots c_{n} >0$, and $\sum_{i=k+1}^{n} c_{i}=n$. Let $M_{n}$ be the set of all possible snake maps. 
We define the action \textbf{crawl} $C:M_{n}\rightarrow M_{n}$ such that for any $c\in M_{n}$, $C(c)=c'$ where,

For $1\leq i \leq n{-}1$,

$$c'_{i}=  \left\{ \begin{array}{ll}

\max\{0, c_{i{+}1}{-}1\}& \mbox{If  } c_{1}, \dots ,c_{i} \leq 1; \\
  c_{i{+}1} & \mbox{otherwise.} \\
\end{array} \right.
$$

And, $c'_{n}=n-\sum_{i=1}^{n-1} c'_{i}$.

\end{defin}
\begin{lemma} \label{WZ}

Consider some arbitrary $\nu$ a permutation of $[n]$. Let $T_{m\times n}$ be a tuple board corresponding to \emph{WZ}$_{\nu}$,  $i$  a row in $T$ containing a  $1$, and snake $s$  starting at row $i$. Let $j$ be the smallest number greater than $i$ containing another $1$, and  $s'$ the snake  starting at row $j$.

\begin{enumerate}

\item We have $j=i+2$.
\item Let $c=(c_{1},\dots , c_{n})$ be snake map of $s$ and $c'=(c'_{1},\dots , c'_{n})$ be snake map of $s'$. We have $c'=C(c)$

\item Any element of the tuple board is either a $0$ or it belongs to a snake. 
\end{enumerate}

\end{lemma}

\noindent
\begin{proof}

For simplicity we assume $\nu=(1,2,\dots ,n)$, and drop subscript from $W$. The proof any arbitrary permutation $\nu$ of ${[n]}$ will be similar. 

\textbf{Proof of Part 1. }
First, we argue that $c_{1}\leq 1$. We know that $\sum_{i=1}^{n}c_{i}=n$. If $c_{1}\neq 0$ then we have
$$c_{2},\dots , c_{n}>0\Rightarrow c_{2}+\dots+c_{n}\geq n{-}1 \Rightarrow c_{1}\leq n-(n{-}1)=1.$$  
Since there is a $1$ in row $i$ if $T(i,1)=0$, $T(i,2)=0$ or $T(i,2)=1$. In both cases, applying winching will derive, $T(i{+}1,1)=0$.
If $T(i,1)=1$ since $c_{1}\leq 1$ we have $T(i,2)=2$, and $T(i{+}1,1)=0$. Thus,  in both cases $T(i{+}1,1)=0$.
Consider  the first column where  $s$ turns down and let it be column $j$.  We have, $c_{1},\dots ,c_{l{-}1}\leq 1$, and $T(i,l)+1=T(i{+}1,l)$. 
Note that  for any $1\leq k \leq l,\quad T(i,k)=0$ or $T(i,k)=T(i,k{-}1)+1$. Hence, applying \emph{WZ} to $y=T(i,.)$ will result in $T(i{+}1,k)=0$ for $1\leq k\leq l{-}1$, and $T(i{+}1,l)>1$. Since the rest of entries in row $i{+}1$ will be in increasing order, there will not be any $1$ in this row.
In the $i+2$nd row, we will have $T(i{+}2,k)=0$ for $1\leq k \leq l{-}2$, and $T(i{+}2,l{-}1)=1$. As a result, $j=i+2.$

\textbf{ Proof of Part 2.}
As we argued in part 1, $s'$ starts at $T(i{+}2,l{-}1)$. Since elements of $s$ and $s'$ will be nonzero from this point to the right then by the same argument  presented  for regular winching  in the proof of Lemma \ref{pp}) we can show the snakes will move in parallel to each other. Note that in row $i{+}1$ we have zeros until we get to $T(i{+}1,l)$. In row $i{+}2$ we have zeros in $T(i{+}2,1),\dots ,T(i{+}2,l{-}2)$, i.e. $c'_{1}=\dots=c'_{l-2}=0$. At column $l{-}1$, $s'$ will start and move parallel to $s$ but there is a zero between $s$ and $s'$ in column $l{-}1$. Therefore, $c'_{l-1}=c_{l}-1$, and there is no zero between the remaining part of $s$ and $s'$, thus $c'_{k}=c_{k{+}1}$, for $l\leq k\leq n-1$. The rest of $s'$ will continue in column $n$. 

\textbf{ Proof of Part 3.}
 It is clear from the above argument that the space between the initial segments of any two snakes is filled with zeros.
\end{proof}

\begin{lemma}\label{crawl}\

For any  $c\in M_{n}$ we have, $C^{n}(c)=c$.

\end{lemma}

We will need the following definitions and lemmas to prove Lemma \ref{crawl}.

\begin{defin} 
Consider the set $M_{n}$.  We define the one-to-one map ${\cal F}:M_{n}\rightarrow \{0,1\}^{2n}$ as follows: 

 for all $c\in M_{n}$, ${\cal F}(c)=b=(b_{1},\dots, b_{2n})$ where for $1\leq i \leq n$,
$$b_{i}=  \left\{ \begin{array}{ll}

 1 & \mbox{If  } \exists k ; c_{1}{+} \dots {+}c_{k} =i ; \\
0 & \mbox{Otherwise.} \\
\end{array} \right.
$$

And, for $n<i\leq 2n$, $b_{i}=\neg b_{i-n{+}1}$. 
\end{defin}
\begin{lemma}{ $\cal F$ is one-to-one. }
\end{lemma}
\begin{proof}
Assume ${\cal F}(x)={\cal F}(y)=w$, and let $j$ be the smallest index where $w_{j}= 1$. We have $x_{1}=y_{1}=j$. The next nonzero index will determine that $x_{2}=y_{2}$ and likewise, we can verify  that all entries of $x$ and $y$ are equal.
\end{proof}
\begin{defin}{
Let ${\cal B}_{n} \subset \{0,1\}^{2n}$ be the set of all $b\in\{0,1\}^{2n}$ such that for all,  $1\leq i \leq n, b_{i}=\neg b_{n+i}$. We define the action of rotation ${\cal R}:{\cal B}\rightarrow {\cal B}$ on this set as follows:
Partition $b$ into maximal blocks of $1^k0$, remove the leftmost block, and put it on the right.

}\end{defin}

\begin{exm}

 Let $b= (110010001101)$. The partitioning of  $b$ will be $(110.0.10.0.0.110.1)$. Therefore, ${\cal R}(b)= (0.10.0.0.110.1.110)$.
\end{exm}

\begin{lemma} \label{Map}
{
For $c\in M_{n}$, we have ${\cal F}( C(c))= {\cal R}({\cal F}(c))$
}\end{lemma}

\begin{proof}

Consider an arbitrary $c=(c_{1},\dots , c_{n})\in M_{n}$. 
Let's say we have $c_{1}=\dots =c_{k-1}=0$, and $c_{k}$ is the leftmost nonzero element in $c$. 
Consider the set ${\cal A}=\{a_{1}=c_{k}, a_{2}=c_{k}{+}c_{k+1},\dots ,a_{n-k}=\sum_{i=k}^{n}c_{i}\}$. Let \emph{C}$(c)=c'$ and $b=(b_{1},\dots,b_{n})$, the binary word representing $\cal A$. In other words for all $a ,  a\in {\cal A}  \Leftrightarrow b_{a}=1$. 

Similarly, let ${\cal A}'=\{c_{k'}, c_{k'}{+}c_{k'+1},\dots ,\sum_{i=k'}^{n}c'_{i}\}$ where $k'$ is the leftmost nonzero element in $c'$ and $b'$ be ${\cal A}'$s binary representation.

According to definition of crawl we know that, if $c_{k}=\dots= c_{l-1}{=}1$, we will have $c'_{k-1},\dots,c'_{l-2}=0$ and $c'_{l-1}=c_{l}{-}1$, where $l$ is the leftmost element greater than $1$. This means that if have $a_1=1, a_{2}=2, a_{3}=3,\dots a_{l-1}=l,$ they should be removed from $\cal A$ to make ${\cal A}'$ . In other words, any set of consecutive elements starting from a $1$ will be removed in ${\cal A}'$. Moreover, $c_l$ will be decremented which means $a_{1}$ and the rest of the elements in $\cal A$ will be decreased by $l$ except the last one which should always be an $n$.  Now, let's see how $b$ will change accordingly.  We remove consecutive elements starting with a $1$ from $\cal A$ which means we remove   the preceding $1$s from $b$ until we hit a $0$.  All the other elements will be decreased by $l$ which means they should be shifted to left by $l$ positions. This is equivalent to removing the first block from $b$. Now, we need to add $b'_{n{-}l{+}1},\dots, b'_{n-1}=0$. And $b'_{n}=1$ because $c'_n$ should be increased by $l$ to make the length of the snake equal to $n$. This whole process is removing the leftmost block and adding its negation to the right, which is equivalent to a rotation of a block in ${\cal F}(c)$. 
\end{proof}

\begin{lemma}\label{getback}
$\forall x\in {\cal B}_{n},\quad {\cal R}^{n}(x)=x.$
\end{lemma}
\begin{proof}
Consider any arbitrary $x$, any block in $x$ has a single $0$. Moreover, the number of zeros in $x$ is $n$. Therefore, after $n$ rotations $x$ will get back to its initial state.  
\end{proof}

\textbf{Proof of Lemma \ref{crawl}.}  From Lemma \ref{getback} and \ref{Map} and the fact that $\cal F$ is a one-to-one function we have, $\forall c\in {\cal M}_{n},\quad C^{n}(c)=n$.

\textbf{Proof of Theorem \ref{zerowinchthm} Part 1.}
 By employing Lemma \ref{WZ} we can verify that the snakes appear in alternating rows. By Lemma \ref{crawl} we know that each snake gets back to itself after $n$ crawls. Thus, $T(1,.)=T(2n{+}1,.)$ where $T$ is a tuple board corresponding to \emph{WZ}, and  {{\emph W}}$Z^{2n}(x)=x$.
 
$\hspace{5.9 in}\Box$

\textbf{Proof of Theorem \ref{zerowinchthm} Part 2.}
Part 2. Using Lemma \ref{WZ} part $3$ we know that half of any tuple board is filled with zeros, and the rest is filled by equal repetitions of numbers $1$ to $n$. In addition, there are $n$ snakes in any tuple board and in any snake $j$ appears once and only once. Therefore, each element will appear $n$ times in the tuple board and the average of $f_{j}=1/2$ for each $j$. 

$\hspace{5.9 in}\Box$

\section*{Acknowledgements}
\label{sec:ack}
This work benefited from helpful discussions with James Propp and Tom Roby. In fact Theorem \ref{winchthm}, the correspondence between winching and promotion/rowmotion and also the homomesic functions discussed in Section \ref{results} were observed and brought to author's attention  by James Propp. I also would like to thank Peter Winkler for his very helpful assistance and comments.


\end{document}